\documentclass[12pt]{amsart}

\usepackage{amssymb}

\usepackage{enumerate}
\usepackage[super]{nth}
\usepackage{enumitem,kantlipsum}

\usepackage{graphicx}
\usepackage{hyperref}

\makeatletter
\@namedef{subjclassname@2010}{
  \textup{2010} Mathematics Subject Classification}
\makeatother

\newtheorem{thm}{Theorem}[section]
\newtheorem{cor}[thm]{Corollary}

\newtheorem{prop}[thm]{Proposition}

\newtheorem{lemma}[thm]{Lemma}
\newtheorem{example}[thm]{Example} 
\newtheorem{definition}[thm]{Definition}
\newtheorem{question}[thm]{Question}

\newtheorem{claim}{Claim}

\theoremstyle{definition}

\numberwithin{equation}{section}

\newcommand{\Q}{\mathbb{Q}}
\newcommand{\Z}{\mathbb{Z}}

\newcommand{\F}{\mathbb{F}}

\newcommand{\OK}{\mathcal{O}_K}
\newcommand{\pbinom}[2]{{#1 \choose #2}_p}

\newcommand{\disc}{\mathop{\rm disc}}

\newcommand{\col}{\mathop{\rm col}}

\newcommand{\es}[1]{\begin{equation}\begin{split}#1\end{split}\end{equation}}

\begin{document}

\baselineskip=17pt

\title[Counting Subrings of $\Z^n$]{Counting Finite Index Subrings of $\Z^n$}

\author[S. Atanasov]{Stanislav Atanasov}
\address{Department of Mathematics, Columbia University, Room 408, MC4406, 2990 Broadway, New York, NY 10027}
\email{stanislav@math.columbia.edu}

\author[N. Kaplan]{Nathan Kaplan}
\address{Department of Mathematics, University of California, Irvine, 340 Rowland Hall, Irvine, CA 92697} 
\email{nckaplan@math.uci.edu}

\author[B. Krakoff]{Benjamin Krakoff}
\address{Department of Mathematics, University of Michigan, East Hall B723, 530 Church Street, Ann Arbor, MI 48109}
\email{bkrakoff@umich.edu}

\author[J. H. Menzel]{Julia H. Menzel}
\address{Program in History, Anthropology, and Science, Technology, and Society, 77 Massachusetts Avenue, Room E51-163, MIT, Cambridge, MA 02139}
\email{julimen@mit.edu}

\date{\today}

\begin{abstract}
We count subrings of small index of $\mathbb{Z}^n$, where the addition and multiplication are defined componentwise. Let $f_n(k)$ denote the number of subrings of index $k$.  For any~$n$, we give a formula for this quantity for all integers $k$ that are not divisible by a \nth{9} power of a prime, extending a result of Liu. 
\end{abstract}

\subjclass[2010]{Primary 20E07; Secondary 11H06, 11M41}

\keywords{Subrings, Lattices, Multiplicative lattices, Zeta functions of groups and rings}

\maketitle

\section{Introduction}\label{sec1}

The main goal of this paper is to study the number of subrings of $\Z^n$ of given index.  We begin by reviewing an easier problem, counting subgroups of $\Z^n$ of given index.

\subsection{Counting Subgroups of $\Z^n$}\label{sec_counting_subgroups}

The \emph{zeta function} of an infinite group $G$ is defined by
\[
\zeta_G(s) = \sum_{H \le G \atop [G:H] < \infty} [G:H]^{-s} = \sum_{k=1}^\infty a_k(G) k^{-s},
\]
where $s$ is a complex variable and $a_k(G)$ is the number of subgroups of $G$ of index $k$.  We can think of $\zeta_G(s)$ as a generating function that gives the number of subgroups $H$ of $G$ of each finite index.  

We focus on the case $G = (\Z^n,+)$.  A finite index subgroup of $\Z^n$ is a sublattice, and every sublattice of $\Z^n$ is the column span of a unique matrix $A$ in Hermite normal form.  The index of the lattice spanned by $A$ is $\det(A)$.  Let $M_n(\Z)$ denote the set of all $n \times n$ matrices with entries in $\Z$.  We have
\[
a_k(\Z^n) = \#\{A \in M_n(\Z) \colon A \text{ is in Hermite normal form and } \det(A) = k\}.
\]

Throughout this paper, $p$ always represents a prime number and $\prod_p$ denotes a product over all primes.  The zeta function of a finitely generated nilpotent group $G$ has an Euler product \cite[Proposition 1.3]{GSS}, so we can write
\[
\zeta_{\Z^n}(s) = \prod_p \zeta_{\Z^n,p}(s),
\]
where 
\[
\zeta_{\Z^n,p}(s) = \sum_{k=0}^{\infty} a_{p^k}(\Z^n) p^{-ks}.
\] 

A matrix $A \in M_n(\Z)$ in Hermite normal form with $\det(A) = p^k$ has diagonal $(p^{i_1}, p^{i_2},\ldots, p^{i_n})$ where each $i_j \ge 0$ and $\sum_{j=1}^n i_j = k$.  It is not difficult to compute the number of $n \times n$ matrices in Hermite normal form with given diagonal. This computation implies that
\begin{equation}\label{local_factor_subgroup}
\zeta_{\Z^n,p}(s) = (1-p^{-s})^{-1} (1-p^{-(s-1)})^{-1}\cdots (1-p^{-(s-(n-1))})^{-1},
\end{equation}
and therefore
\begin{equation}\label{subgroupeqn}
\zeta_{\Z^n}(s) = \zeta(s) \zeta(s-1) \cdots \zeta(s-(n-1)).
\end{equation}
See the book of Lubotzky and Segal for five proofs of this fact \cite{LubotzkySegal}.  We review one of these arguments in Section \ref{sec_hermite}, as it forms the basis for the approach to counting subrings that we explain in Section \ref{sec_subring_matrices}.

The function in \eqref{subgroupeqn} has meromorphic continuation to the entire complex plane.  Its right-most pole is located at $s=n$ and is simple.  A standard Tauberian theorem gives an asymptotic formula for the number of sublattices of $\Z^n$ of bounded index.  We have
\es{\label{asymptotic_subgroup_count}
N_n(X)  := &\   \#\{\mbox{Sublattices of } \Z^n
\mbox{ of index} <X\}  = \sum_{k < X} a_k(\Z^n) \\
 =  &\  \frac{\zeta(n)\zeta(n-1)\cdots \zeta(2)}{n}X^n+O(X^{n-1}\log(X))
}
as $X\to\infty$.  

In Section \ref{sec_hermite} we see that for fixed $n$ and $e,\ a_{p^e}(\Z^n)$ is a polynomial in $p$ that is not difficult to compute.  In fact, $a_{p^e}(\Z^n)$ is equal to the $p$-binomial coefficient $\pbinom{n-1+e}{e}$ \cite[Chapter 1.8]{Stanley}.  Therefore, the problems of counting sublattices of $\Z^n$ of given index, and of asymptotically counting sublattices of bounded index, are well-understood.

\subsection{Counting Subrings of $\Z^n$}\label{counting_subrings}

We study the function analogous to $a_k(\Z^n)$ that counts subrings of $\Z^n$. We use the term \emph{subring} to mean a multiplicatively closed sublattice containing the multiplicative identity $(1,1,\ldots, 1)$.  Let $f_n(k)$ denote the number of subrings of $\Z^n$ of index $k$. Define the \emph{subring zeta function of $\Z^n$} by
\[
\zeta^R_{\Z^n}(s) = \sum_{k=1}^\infty f_n(k) k^{-s}.
\]
As we saw with $\zeta_{\Z^n}(s)$, this zeta function has an Euler product
\[
\zeta^R_{\Z^n}(s) = \prod_p \zeta_{\Z^n,p}^R(s),
\]
where
\[
\zeta_{\Z^n,p}^R(s) = \sum_{k=0}^\infty f_n(p^k) p^{-ks}.
\]
Equivalently, $\zeta_{\Z^n,p}^R(s)  = \zeta_{\Z_p^n}^R(s)$ where $\Z_p$ denotes the ring of $p$-adic integers and $\zeta^R_{\Z_p^n}(s)$ is the zeta function that counts finite index $\Z_p$-subalgebras of~$\Z_p^n$. 

\begin{question}
For fixed $n$ and $e$, how does $f_n(p^e)$ behave as a function~of~$p$?
\end{question} 

Liu uses a strategy similar to the one outlined in Section \ref{sec_counting_subgroups} to compute $f_n(p^e)$ for $e \le 5$ and any $n$.  (There is a small error in the computation of $f_n(p^5)$ that we correct here. More specifically, the constant terms in the coefficients of ${n \choose 6}$ and ${n \choose 7}$ are corrected to $141$ and $371$, respectively.)
\begin{prop}\cite[Proposition 1.1]{Liu}\label{Liu_fn_prop}
We have
\begin{eqnarray*}
f_n(1) & = & 1, \\
f_n(p) & = & \binom{n}{2},\\
f_n(p^2) & = & \binom{n}{2} + \binom{n}{3} + 3 \binom{n}{4},\\
f_n(p^3) & = & \binom{n}{2} + (p+1) \binom{n}{3} + 7 \binom{n}{4} + 10 \binom{n}{5} + 15 \binom{n}{6}, \\
f_n(p^4) & = & \binom{n}{2} + (3p+1) \binom{n}{3} + (p^2+p+10) \binom{n}{4} + (10p+21) \binom{n}{5} \\
& & + 70 \binom{n}{6} + 105 \binom{n}{7} + 105 \binom{n}{8},\\
f_n(p^5) & = &  {n \choose 2} + (4p + 1){n \choose 3} + (7p^2 + p + 13){n \choose 4}  \\
& & + (p^3 + p^2 + 41p + 31){n \choose 5} + (15p^2 + 35p + 141){n \choose 6}  \\
 & &  +(105p + 371){n \choose 7}+ 910{n \choose 8} + 1260{n \choose 9}+ 945{n \choose 10}.
\end{eqnarray*}
\end{prop}

The main theorem of this paper extends Proposition \ref{Liu_fn_prop} to all $e\le 8$.  For consistency with Liu's results, we state our formulas as linear combinations of binomial coefficients $\binom{n}{i}$ with coefficients that are polynomials in $p$.  This also allows one to quickly identify the main term of each function when $n$ is fixed and $p$ varies.
\begin{thm}\label{Main_Theorem}
We have 
\begin{footnotesize}
 \begin{eqnarray*}
 f_n(p^6) & = & {n \choose 2} + (p^2 + 4p + 1){n \choose 3} + (p^3 + 14p^2 + p + 16){n \choose 4} \\
 & & + (p^4+11p^3+2p^2+81p+41){n \choose 5} + (p^4 + p^3 + 131p^2 + 111p + 226){n \choose 6} \\
 & & + (21p^3 + 56p^2 + 616p + 743){n \choose 7}+ (210p^2 + 770p + 2639){n \choose 8} \\ 
 & & + (1260p + 6958){n \choose 9} + 14175{n \choose 10}+ 17325{n \choose 11} + 10395{n \choose 12},\\
 & & \\
 & & \\
f_n(p^7) & = & {n \choose 2} + (3p^2 + 4p + 1){n \choose 3} + (10p^3 + 12p^2 + p + 19){n \choose 4}\\
    & & + (15p^4+21p^3+16p^2+121p+51){n \choose 5} \\
    & & + (p^6 + p^5 + 17p^4 + 17p^3 + 392p^2 + 206p + 326){n \choose 6}\\
    & & +  (p^5 + 22p^4 + 288p^3 + 379p^2 + 1618p + 1219){n \choose 7} \\
    & & + (28p^4 + 84p^3 + 2324p^2 + 3640p + 5279){n \choose 8}\\
    & & +  (378p^3 + 1638p^2 + 11298p + 18600){n \choose 9}+(3150p^2 + 15750p + 58800){n \choose 10}\\
    & & + (17325p + 143605){n \choose 11} + 252945{n \choose 12} + 270270{n \choose 13} + 135135{n \choose 14}, 
\end{eqnarray*}
and 
\begin{eqnarray*}
f_n(p^8) & = & {n \choose 2} + (4p^2 + 4p + 1){n \choose 3} + (p^4 + 26p^3 + 9p^2 + p + 22){n \choose 4} \\
& & + (p^5 + 77p^4 -13p^3 + 52p^2 + 161p + 61){n \choose 5} \\
& & + (16p^6 + 31p^5 + 22p^4 + 187p^3 + 702p^2 + 301p + 441){n \choose 6} \\
& & + (p^8 + p^7 + 2p^6 + 23p^5 + 339p^4 + 1080p^3 + 1206p^2 + 3074p + 1800){n \choose 7}\\
& & + (29p^6 + 29p^5 + 652p^4 + 1093p^3 + 9374p^2 + 9073p + 8933){n \choose 8}\\
& & + (36p^5 + 498p^4 + 6420p^3 + 15324p^2 + 39810p + 37201){n \choose 9}\\
& & + (630p^4 + 3150p^3 + 46200p^2 + 103320p + 148551){n \choose 10}\\
& & + (6930p^3 + 41580p^2 + 243705p + 510730){n \choose 11} \\
& & + (51975p^2 + 329175p + 1474165){n \choose 12}\\
& & + (270270p + 3258255){n \choose 13} + 5045040{n \choose 14} + 4729725{n \choose 15} + 2027025{n \choose 16}.
\end{eqnarray*}
\end{footnotesize}
\end{thm}

\subsection{Motivation: Counting Subrings and Orders of Bounded Index}

Bhargava has asked about the asymptotic growth rate of $f_n(k)$ \cite{Liu}.  We would like to have an asymptotic formula for the number of subrings of $\Z^n$ analogous to \eqref{asymptotic_subgroup_count}.  Expressions for $\zeta^R_{\Z^n}(s)$ analogous to \eqref{subgroupeqn} would lead to such results.  However, such formulas are only known for $n \le 4$.
\begin{thm}\label{GeneratingFunctions}
We have 
\begin{eqnarray*}
\zeta_{\Z^2}^R(s) & = & \zeta(s), \\
\zeta_{\Z^3}^R(s) & = & \frac{\zeta(3s-1) \zeta(s)^3}{\zeta(2s)^2}, \\
\zeta_{\Z^4}^R(s) & = & \prod_p
\frac{1}{(1-p^{-s})^2(1-p^2 p^{-4s})(1-p^3 p^{-6s})} \Big(1+4 p^{-s}+2 p^{-2s}\\
& & +(4p-3) p^{-3s}+(5p-1)p^{-4s}  +(p^2-5p)p^{-5s}+(3p^2-4p) p^{-6s} \\ 
& & -2p^2 p^{-7s}-4p^2 p^{-8s} - p^2 p^{-9s}{\Big)}.
\end{eqnarray*}
\end{thm}
\noindent The computation for $n=2$ is elementary.  The $n=3$ result is originally due to Datskovsky and Wright \cite{DatskovskyWright}, and for $n =4$ it is a result of Nakagawa \cite{Nakagawa}.  Liu gives combinatorial proofs of these results \cite{Liu}; his $A_n(p,p^{-s})$ is $\zeta_{\Z^n,p}^R(s)$.

Kaplan, Marcinek, and Takloo-Bighash study the problem of counting subrings of bounded index in $\Z^n$ and prove the following.
\begin{thm}\cite[Theorem 6]{KMTB}\label{KMTB_Theorem}
Let 
\[
N^R_n(X)  :=   \#\{\mbox{Subrings of } \Z^n
\mbox{ of index less than } X\}  = \sum_{k < X} f_n(k).
\]
\begin{enumerate}
\item Let $n \le 5$.  There is a positive real number $C_n$ such that
\[
N^R_n(X) \sim C_n X (\log X)^{\binom{n}{2}-1}
\]
as $X \to \infty$.

\item Suppose $n \ge 6$.  Then for any $\epsilon > 0$ we have 
\[
X (\log X)^{\binom{n}{2} -1} \ll N^R_n(X) \ll_{\epsilon} X^{\frac{n}{2}-\frac{7}{6} + \epsilon}.
\] 

\end{enumerate}
\end{thm}
The authors of \cite{KMTB} derive the asymptotic order of growth for $N^R_5(X)$ up to a constant factor, despite not having a formula for $\zeta^R_{\Z^5}(s)$ analogous to those of Theorem \ref{GeneratingFunctions}.  The main idea is to find the location and order right-most pole of $\zeta_{\Z^n}^R(s)$ by computing $f_n(p^e)$ exactly for small $e$ and giving estimates for larger $e$.  A major motivation for the computations of this paper is to try to prove stronger versions of Theorem \ref{KMTB_Theorem}.  For $n \ge 6$ we do not even know of a conjecture for the asymptotic growth rate of $N_n^R(X)$.

One of the main problems in the field of arithmetic statistics is to count finite extensions of a number field and the orders that they contain.  For example, it is an old conjecture that the number of isomorphism classes of degree $n$ extensions $K$ of $\Q$ with $|\disc(K)| < X$ is asymptotic to a constant depending on $n$ times $X$.  One can also ask for the number of isomorphism classes of orders contained in these fields with discriminant at most $X$ in absolute value.  Bhargava has proven breakthrough results counting quartic and quintic fields by first counting all isomorphism classes of orders in these fields and then sieving for the maximal ones \cite{Bhargava4, Bhargava5}.  

Bhargava, Malle, and others have made conjectures for counting finite extensions with bounded discriminant and specified Galois group \cite{BhargavaIMRN, Malle1, Malle2}.  Problems about counting orders contained in field extensions of given degree with bounded discriminant have received less attention.  Recall that if $K$ is a number field with ring of integers $\OK$, an order $\mathcal{O} \subseteq \OK$ is a subring of $\OK$ with identity that is a $\Z$-module of rank $n$.  If $\mathcal{O} \subseteq \OK$ is an order, then $\disc(\mathcal{O}) = [\OK:\mathcal{O}]^2\cdot \disc(\OK)$.
\begin{question}
Let $B_n(X)$ denote the number of isomorphism classes of orders $\mathcal{O}$ in all degree $n$ number fields such that $|\disc(\mathcal{O})| < X$.  How does $B_n(X)$ grow as a function of $X$?
\end{question}
It follows from work of Davenport and Heilbronn for $n = 3$ \cite{DavenportHeilbronn}, and Bhargava for $n = 4,5$ \cite{Bhargava4, Bhargava5}, that $B_n(X)$ is asymptotic to a constant $c_n$ times $X$.  For $n \ge 6$ we do not know of a conjecture for the asymptotic growth rate of this function.  

One approach to this problem is to count orders contained in a fixed field $K$ of bounded index, and take a sum over all $K$ of fixed degree.  
\begin{question}\label{count_orders}
Let $K$ be a number field and let $N_K(X)$ denote the number of isomorphism classes of orders $\mathcal{O}$ contained in $K$ such that $|\disc(\mathcal{O})| < X$.  How does $N_K(X)$ grow as a function of $X$? 
\end{question}

Kaplan, Marcinek, and Takloo-Bighash study Question \ref{count_orders} by investigating analytic properties of the subring zeta function of $\OK$.  Let 
\[
\zeta_{\OK}^R(s) = \sum_{\mathcal{O} \subseteq \OK} [\OK:\mathcal{O}]^{-s}
\]
where the sum is taken over all orders in $\OK$.

In the statement of the theorem below, $n = [K:\Q]$ and $r_2$ is an explicitly computable positive integer that depends on the Galois group of the normal closure of $K/\Q$; see \cite{KMTB} for details.
\begin{thm}\cite[Theorem 2]{KMTB}\label{KMTB_Theorem_Orders}
\begin{enumerate}
\item For $n \le 5$, there is a constant $C_K >0$ such that 
\[
N_K(X) \sim C_K X^{1/2} (\log X)^{r_2 -1},
\]
as $X \to \infty$,

\item For any $n \ge 6$ and any $\epsilon > 0$, 
\[
X^{1/2} (\log X)^{r_2 -1} \ll N_K(X) \ll_\epsilon X^{\frac{n}{4} - \frac{7}{12} + \epsilon}.
\]
\end{enumerate}
\end{thm}
\noindent When $[K:\Q] \ge 6$, we do not know of a conjecture for the asymptotic growth rate of $N_K(X)$.

The subring zeta function of $\OK$ has an Euler product indexed by rational primes $p$, and its local factors satisfy $\zeta^R_{\mathcal{O}_K,p}(s) = \zeta^R_{\mathcal{O}_K \otimes \Z_p}(s)$, where this zeta function counts finite index $\Z_p$-subalgebras.  When $p$ splits completely in $\OK,\ \mathcal{O}_K \otimes \Z_p \cong \Z_p^n$, so 
\[
\zeta_{\mathcal{O}_K,p}^R(s) = \zeta_{\Z_p^n}^R(s) =  \zeta_{\Z^n,p}^R(s).
\]
The $n =3$ case of Theorem \ref{KMTB_Theorem_Orders} follows from work of Datskovsky and Wright \cite{DatskovskyWright}, who compute $\zeta_{\OK}^R(s)$ for any cubic field $K$.  The $n=4$ case follows from Nakagawa's computation for any quartic field $K$ of $\zeta^R_{\OK,p}(s)$ at all unramified primes $p$ \cite{Nakagawa}.  

The authors of \cite{KMTB} suggest that among all unramified primes, those that split completely may control the asymptotic growth rate of $N_K(X)$.  This suggests that the growth rate of the simpler function $N_n^R(X)$ along with the Galois group of the normal closure of $K$ may determine the growth rate of $N_K(X)$. For more information, see the discussion following  \cite[Theorem 4]{KMTB}.  We hope that more precise results on the growth of $f_n(p^e)$ will lead not only to improved asymptotic estimates for counting subrings of $\Z^n$ of bounded index, like those of Theorem \ref{KMTB_Theorem}, but may also help to understand asymptotic formulas for counting orders of bounded index in a fixed number field, like those of Theorem \ref{KMTB_Theorem_Orders}.  

In Section \ref{sec_lower_bounds} we give lower bounds for $f_n(p^e)$ that are analogous to lower bounds for $N_K(X)$ due to Brakenhoff \cite{Brakenhoff}, suggesting a closer connection between these counting problems.

\subsection{Motivation: Uniformity of Subring Zeta Functions}

Zeta functions of infinite groups, rings, and algebras have been studied extensively from both a combinatorial and analytic point of view \cite{duSautoyWoodward, Rossmann2, Rossmann3, Voll}.  A common question is how local factors of zeta functions vary with $p$. 
\begin{definition}\label{uniform_def}
A zeta function $\zeta_G(s) = \prod_p \zeta_{G,p}(s)$ is \emph{finitely uniform} if there are finitely many rational functions $W_i(X,Y) \in \Q(X,Y),\ i \in I$, a finite index set, such that for every prime $p$ there exists an $i=i(p)$ such that $\zeta_{G,p}(s) = W_i(p,p^{-s})$.  We say that $\zeta_G(s)$ is \emph{uniform} if it is finitely uniform for $|I| = 1$.
\end{definition}
\noindent Definition \ref{uniform_def} can be extended in the obvious way to subring zeta functions.

\begin{example}\label{local_factor_ex}
\begin{enumerate}[wide, labelwidth=!, labelindent=0pt]  
\item In \eqref{local_factor_subgroup} we saw that $\zeta_{\Z^n,p}(s)$ is given by a single rational function in $p$ and $p^{-s}$.  That is, $\zeta_{\Z^n}(s)$ is uniform.  
\item When $K$ is a number field of degree at most $4$ with ring of integers $\OK,\ \zeta^R_{\mathcal{O}_K}(s)$ is finitely uniform.  For $[K\colon \Q] = 3,4,\ \zeta^R_{\mathcal{O}_K,p}(s)$ depends on the decomposition of the ideal generated by $p$ in $\OK$ \cite{DatskovskyWright,Nakagawa}.
\end{enumerate}
\end{example}

In order to understand how $f_n(p^e)$ varies with $p$, we want to know how the local factors $\zeta_{\Z^n,p}^R(s)$ vary.  Grunewald, Segal, and Smith build on work of Denef \cite{Denef1,Denef2}, and Igusa \cite{Igusa}, to prove the following result.
\begin{thm}\label{GSSthm}\cite[Theorem 3.5]{GSS}
For each positive integer $n$ and each prime $p$ there exist polynomials $\Phi_{n,p}(X), \Psi_{n,p}(X) \in \Z[x]$ such that
\[
\zeta_{\Z^n,p}^R(s) = \frac{\Phi_{n,p}(p^{-s})}{\Psi_{n,p}(p^{-s})}.
\]
Moreover, the degrees of $\Phi_{n,p}$ and $\Psi_{n,p}$ are bounded independently of $p$.
\end{thm}

When $n$ is fixed, we want to understand how these rational functions vary with $p$.
\begin{question}[Question 3.7, \cite{VollRecent}]\label{zetaznuniform}
Is the zeta function $\zeta_{\Z^n}^R(s)$ uniform?  Is it finitely uniform?
\end{question}

Expanding the rational functions of Theorem \ref{GSSthm} as power series and computing individual coefficients shows that $\zeta_{\Z^n}^R(s)$ is uniform if and only if for each fixed $e\ge 1,\ f_n(p^e)$ is a polynomial in~$p$.  Therefore, the first part of Question \ref{zetaznuniform} is equivalent to the following question.

\begin{question}\label{fnpolyq}
For fixed $n \ge 2$ and $e \ge 1$, is $f_n(p^e)$ a polynomial in $p$?
\end{question}

Proposition \ref{Liu_fn_prop} shows that for for fixed $n$ and $e\le 5,\ f_n(p^e)$ is a polynomial in $p$. Theorem \ref{Main_Theorem} extends this to $e \le 8$ and provides some evidence for a positive answer to Questions \ref{zetaznuniform} and \ref{fnpolyq}.  We will see how the proof of Theorem \ref{Main_Theorem} involves counting $\F_p$-points on certain varieties.  We will see that the variety $V$ in $5$-dimensional affine space over $\F_p$ defined by
\[
(x^2-x) - (u^2 - u) c'  =(y^2-y) - (v^2 - v) c' = xy -uv c'  = 0
\]
plays an important role in the proof of Lemma \ref{lemma3211}.  A computation in Magma shows that $V$ is $2$-dimensional, with $7$ irreducible components, and suggests that $\#V(\F_p) = 7p^2 - 6p+6$.  We verify this formula in the proof of Lemma \ref{lemma3211}. It is not difficult to imagine how for larger values of $e$, more complicated varieties may play a role in the formula for $f_n(p^e)$.

\subsection{Outline of the Paper}

In Section \ref{sec2} we follow the method of Liu and give a bijection between subrings of $\Z^n$ and matrices of a specific form.  This transforms the question of counting subrings of prime power index into a problem of counting matrices with entries satisfying certain divisibility conditions.  We then review Liu's notion of irreducible subrings. Let $g_n(k)$ be the number of irreducible subrings of $\Z^{n}$ of index $k$. We recall a recurrence due to Liu relating $g_n(k)$ and $f_n(k)$ in Proposition \ref{fnrecurrence}.  Our $g_n(k)$ is denoted by $g_{n-1}(k)$ in \cite{Liu}.

In Section \ref{sec3} we express $g_n(p^e)$ as a sum over irreducible subrings with fixed diagonal entries.   Possible diagonals are in bijection with compositions of $e$ into $n-1$ parts.  We verify that for $n \le 9$ and $e \le 8$ each of the functions counting irreducible subring matrices with a fixed diagonal is given by a polynomial in $p$.  We use these results along with the recurrence of \cite{Liu}  to compute $f_n(p^e)$ for $e \leq 8$, proving Theorem \ref{Main_Theorem}.  In Section \ref{sec_lower_bounds} we give lower bounds for $g_n(p^e)$ analogous to results of Brakenhoff for orders in a fixed number field \cite{Brakenhoff}.  We end with questions for further study.

Computations related to this project were carried out in Sage, Magma, and Mathematica \cite{Magma,Sagemath}.  Some programs are available at the second author's website: {\small\url{https://www.math.uci.edu/~nckaplan/research_files/subrings}}.

\section{Subring Matrices and Irreducible Subring Matrices}
\label{sec2}
In our analysis of $f_{n}(k)$ we employ techniques developed in \cite{Liu}, where Liu gives a bijection between subrings $\mathbb{Z}^n$ and a class of integer matrices. This reduces the problem of counting subrings of index $k$ in $\Z^n$ to the problem of counting \emph{subring matrices}, which can be understood as compositions of yet simpler \emph{irreducible subring matrices}. 

\subsection{Counting Matrices in Hermite Normal Form}\label{sec_hermite}

We begin by giving a proof of \eqref{local_factor_subgroup} due to Bushnell and Reiner \cite[Section 15.2]{LubotzkySegal}.
\begin{definition}
A matrix $A \in M_n(\Z)$ with entries $a_{ij}$ is in \emph{Hermite normal form} if:
\begin{enumerate}
\item $A$ is upper triangular, and
\item $0\leq a_{ij} < a_{ii}$ for $1\leq i < j \leq n$.
\end{enumerate}
\end{definition}

There is a bijection between sublattices of $\Z^n$ of index $p^k$ and matrices $A\in M_n(\Z)$ in Hermite normal form with $\det(A) = p^k$.  The diagonal entries of such a matrix are of the form $(p^{i_1},\ldots, p^{i_n})$, where each $i_j \ge 0$ and $\sum_{j=1}^n i_j = k$.
\begin{definition}
A \emph{weak composition} of an integer $k$ is a list of non-negative integers $(\alpha_1,\ldots, \alpha_n)$ where $\sum_{i=1}^n \alpha_i = k$.  Each $\alpha_i$ is a \emph{part} of the weak composition, and $n$ is the \emph{length} or \emph{number of parts}.  A weak composition in which every part is positive is called a \emph{composition} of $k$.
\end{definition}
\noindent The possible diagonals of an $n \times n$ matrix in Hermite normal form with determinant $p^k$ are in bijection with weak compositions of $k$ of length~$n$.  

The number of $n \times n$ matrices in Hermite normal form with diagonal $(p^{i_1},\ldots, p^{i_n})$ is $p^{(n-1) i_1} p^{(n-2) i_2} \cdots p^{i_{n-1}}$.  Taking a sum of these terms over all weak compositions of $k$ into $n$ parts gives a polynomial formula for $a_{p^k}(\Z^n)$.  We have
\begin{eqnarray*}
\zeta_{\Z^n,p}(s) & = &  \sum_{k=0}^\infty a_{p^k}(\Z^n) p^{-ks} = \sum_{i_1 = 0}^\infty \cdots \sum_{i_n = 0}^\infty p^{-(i_1+\cdots + i_n)s} p^{(n-1) i_1} p^{(n-2) i_2} \cdots p^{i_{n-1}} \\
& = & \left(\sum_{i_1 = 0}^\infty p^{-(s-(n-1))i_1} \right) \cdots \left(\sum_{i_{n-1} = 0}^\infty p^{-(s-1)i_{n-1}} \right) \left(\sum_{i_n = 0}^\infty p^{-s  i_n} \right) \\
& = &  (1-p^{-s})^{-1} (1-p^{-(s-1)})^{-1}\cdots (1-p^{-(s-(n-1))})^{-1},
\end{eqnarray*}
completing the proof of \eqref{local_factor_subgroup}.

\subsection{Counting Subrings via Liu's Bijection}\label{sec_subring_matrices}

Liu adapts the argument of Section \ref{sec_hermite} to count subrings of $\Z^n$.  For column vectors $u = (u_1,\ldots, u_n)$ and $w = (w_1,\ldots, w_n)$ we write $u\circ w$ for the column vector given by the componentwise product $(u_1 w_1,\ldots, u_n w_n)$.
\begin{prop}[Propositions 2.1 and 2.2 in \cite{Liu}]
\label{prop:bijection-between-rings-and-matrices}
Let $k$ be a positive integer.  There is a bijection between subrings with identity $L \subset \mathbb{Z}^{n}$ of index $k$ and matrices $A\in M_n(\Z)$ in Hermite normal form with $\det(A) = k$ such that:
\begin{enumerate}
\item the identity element $(1, \ldots, 1)^{T}$ is in the column span of $A$, and 
\item for each $i,j \in [1,n]$, $v_i \circ v_j$ is in the lattice spanned by the column vectors $v_1,\ldots, v_n$.
\end{enumerate}
\end{prop}

\begin{definition}
\begin{enumerate}[wide, labelwidth=!, labelindent=0pt]  
\item A lattice $L \subset \Z^n$ for which $u,w \in L$ implies $u \circ w \in L$ is \emph{multiplicatively closed}.  
\item A matrix that satisfies the conditions of Proposition \ref{prop:bijection-between-rings-and-matrices} is a \emph{subring matrix}. 
\end{enumerate}
\end{definition}
\noindent A subring matrix $A$ is required to be in Hermite normal form, so $(1,\ldots,1)^T$ is in the column span of $A$ if and only if the final column of $A$ is $(1,\ldots,1)^T$.

For fixed $n$, we may calculate $f_n(k)$ by counting the corresponding subring matrices. Since $f_n(k)$ is weakly multiplicative, it suffices to consider $k=p^{e}$ for $p$ prime. This restricts our attention to subring matrices with diagonal entries $(p^{\alpha_{1}},\ldots,p^{\alpha_{n}})$ such that $(\alpha_1,\ldots, \alpha_n)$ is a weak composition of $e$ of length $n$. 

Liu notes that there is a natural correspondence between subrings of $\Z_p^n$ and subrings of $\Z^n$ with index a power of $p$ \cite[Page 283]{Liu}.  Liu defines \emph{irreducible subrings} of $\Z_p^n$, but for notational convenience we prefer to give the corresponding notion for subrings of $\Z^n$ with index equal to a power~of~$p$.
\begin{definition}\label{irred-def}
A subring $L\subset \Z^{n}$ with index equal to a power of $p$ is irreducible if for each $(x_{1},\ldots,x_{n})\in L,\ x_{1}\equiv x_{2} \equiv \ldots \equiv x_{n} \pmod{p}$.
\end{definition}

\begin{thm}\cite[Theorem 3.4]{Liu}
A subring $L\subset \Z^{n}$ of index equal to a power of $p$ can be written uniquely as a direct sum of irreducible subrings $L_{i}\subset \Z^{n}$.
\end{thm}
It is easy to see if a subring is irreducible by considering the corresponding subring matrix.

\begin{prop}\cite[Proposition 3.1]{Liu}
\label{prop:irreducible-corresponding-matrix}
An $n\times n$ subring matrix represents an irreducible subring if and only if its first $n-1$ columns contain only entries divisible by $p$, and its final column is equal the identity $(1, \ldots, 1)^{T}$.
\end{prop}

Recall that $g_n(k)$ is the number of irreducible subrings of $\Z^n$ of index $k$.  In Section \ref{sec3}, we give formulas for $g_n(p^e)$ for each $e \le 8$.  We recall the following recurrence due to Liu.  We again emphasize that $g_n(p^e)$ is given by $g_{n-1}(p^e)$ in \cite{Liu}.  Define $f_0(1) = 1$ and $f_0(p^e) = 0$ for $e> 0$.

\begin{prop}\cite[Proposition 4.4]{Liu}\label{fnrecurrence}
The following recurrence holds for $n > 0$:
\[
f_n(p^e) = \sum_{i=0}^e \sum_{j=1}^{n} \binom{n-1}{j-1} f_{n-j}(p^{e-i}) g_j(p^i).
\]
\end{prop}
In order to show for fixed $n$ and $e$ that $f_n(p^e)$ is a polynomial in $p$, it is enough to show that $f_j(p^k)$ is a polynomial in $p$ for each fixed $j\le n-1$ and $k \le e$, and that $g_j(p^i)$ is a polynomial in $p$ for each fixed $j \le n$ and $i \le e$.

\section{Computing $f_n(p^6)$, $f_n(p^7)$, and $f_n(p^8)$}\label{sec3}

In this section we compute $f_n(p^6),\ f_n(p^7)$, and $f_n(p^8)$, proving Theorem \ref{Main_Theorem}.  We do this by computing $g_n(p^e)$ for any $n$ and fixed $e\le 8$ and then applying Proposition \ref{fnrecurrence}. 

\subsection{Statement of Results}  

In this subsection, we state the main results that lead to the proof of Theorem \ref{Main_Theorem}. Recall that for fixed $e$, $ g_n(p^e) = 0$ for all but finitely many $n$.  Note that $g_1(1) = 1$ and $g_1(p^e) = 0$ for $e \ge 1$.

\begin{prop}\cite[Proposition 4.3]{Liu}\label{Propgnpn}
For all $n \ge 2$, we have that $g_n(p^e) = 0$ for $e < n-1,\ g_{n}(p^{n-1}) = 1$, and $g_n(p^{n}) = \frac{p^{n-1}-1}{p-1}$.
\end{prop}

Proposition \ref{Liu_fn_prop} gives Liu's polynomial formulas for $f_n(p^e)$ for $e \le 5$ \cite[Proposition 1.1]{Liu}.  We note that there is a slight error in Liu's computation for $e = 5$, so we have stated the corrected result.

It is easy to see that $g_2(p^e) = 1$ for all $e \ge 1$.  Although they are not written explicitly, Liu gives polynomial formulas for $g_n(p^e)$ for $n =3, 4$ \cite[Propositions 6.2 and 6.3]{Liu}.  We record the values of $g_3(p^e)$ for $e \in \{4,5,6,7,8\}$ and of $g_4(p^e)$ for $e\in \{5,6,7,8\}$ since these values are used in the proof of Theorem \ref{Main_Theorem}:

\begin{table}[h!]\label{Tab1}
\renewcommand{\arraystretch}{1.2}
\centering
\begin{tabular}{|c | c |}
\hline
$g_3(p^4)$ & $3p+1$ \\
$g_3(p^5)$ & $4p+1$\\
$g_3(p^6)$ & $p^2+4p+1$ \\
$g_3(p^7)$ & $3p^2+4p+1$\\
$g_3(p^8)$ & $4p^2+4p+1$\\
$g_4(p^5)$ & $7p^2+p+1$\\
$g_4(p^6)$ & $p^3 + 14p^2 + p + 1$\\
$g_4(p^7)$ & $10p^3 + 12p^2 + p + 1$\\
$g_4(p^8)$ & $p^4 + 26p^3 + 9p^2 + p + 1$\\
\hline
\end{tabular}.
\caption{{Values of $g_n(p^e)$ for $n = 3,4$ and $e \le 8$.}}
\end{table}

\vspace{-.7cm}

Combining Propositions \ref{Liu_fn_prop} and \ref{fnrecurrence} with an induction argument gives polynomial formulas for $f_n(p^e)$ for all $e \le 5$ and all $n$.  Theorem \ref{Main_Theorem} follows from Proposition \ref{Liu_fn_prop}, the results of Table 1, and the following result.

\begin{thm}\label{gnpe_thm}
We have the following values of $g_n(p^e)$:
\begin{table}[h!]
\renewcommand{\arraystretch}{1.2}
\centering
\begin{tabular}{|c | c |}
\hline
$g_5(p^6)$ & $p^4 + 11p^3 + 2p^2 + p + 1$ \\
$g_5(p^7)$ & $15p^4 + 21p^3 + 6p^2 + p + 1$\\
$g_5(p^8)$ & $p^5 + 77p^4 - 13p^3 + 12p^2 + p + 1$ \\
$g_6(p^7)$ & $p^6 + p^5 + 17p^4 + 2p^3 + 2p^2 + p + 1$\\
$g_6(p^8)$ & $16p^6 + 31p^5 + 22p^4 + 22p^3 + 2p^2 + p + 1$\\
$g_7(p^8)$ & $p^8 + p^7 + 2p^6 + 23p^5 + 3p^4 + 2p^3 + 2p^2 + p + 1$\\
\hline
\end{tabular}.
\caption{{Values of $g_n(p^e)$ for $n = 5,6,7$ and $e \le 8$.}}
\end{table}
\end{thm}

\subsection{Strategy of the proof of Theorem \ref{gnpe_thm}}

In the rest of this section we explain the strategy of the proof of Theorem \ref{gnpe_thm}.  Recall that the last column of any irreducible subring matrix is $(1,\ldots, 1)^T$ and that every other entry is divisible by $p$.  The possible diagonals of an $n \times n$ irreducible subring matrix of determinant $p^e$ are in bijection with compositions of $e$ of length~$n-1$.  
\begin{definition}
Let $\mathcal{C}_{n,e}$ denote the set of compositions of $e$ into $n-1$ parts.  For a composition $\alpha$ of length $n-1$ let $g_\alpha(p)$ denote the number of irreducible subrings of $\Z^n$ with diagonal entries $(p^{\alpha_1}, p^{\alpha_2}, \ldots, p^{\alpha_{n-1}}, 1)$.  
\end{definition}
\noindent It is a standard fact that $|\mathcal{C}_{n,e}| = \binom{e-1}{n-2}$. 

Combining these definitions shows that for any positive integers $n$ and~$e$,
\begin{equation}
g_n(p^e) = \sum_{\alpha \in \mathcal{C}_{n,e}} g_\alpha(p).
\end{equation}
The following result shows that when computing $g_n(p^e)$ for the values of $n,e$ of Theorem \ref{gnpe_thm} we do not have to compute $g_\alpha(p)$ for every $\alpha \in \mathcal{C}_{n,e}$.

\begin{lemma}\label{lem2}
Let $\alpha = (1,\alpha_2,\ldots, \alpha_k)$ be a composition of a positive integer $e$ and $\alpha' = (\alpha_2,\ldots, \alpha_k)$. We have $g_\alpha(p) = g_{\alpha'}(p)$.
\end{lemma}

\begin{proof}
An irreducible subring matrix with a $p$ as its first diagonal entry has first row equal to $(p,0,\ldots, 0,1)$ since every entry $a_{1,j}$ with $j \not\in \{1,n\}$ satisfies $0 \le a_{1,j} < p$ and $a_{1,j} \equiv 0 \pmod{p}$.  The conditions derived from taking products of pairs of columns are identical in both cases.
\end{proof}

Let $\mathcal{C}'_{n,e}$ denote the set of compositions of $e$ into $n-1$ parts where the first part is larger than $1$.  Lemma \ref{lem2} implies that 
\begin{equation}\label{gn_recursive}
g_n(p^e) = g_{n-1}(p^{e-1}) + \sum_{\alpha \in \mathcal{C}_{n,e}'} g_\alpha(p).
\end{equation}
Using this idea, the formula for $g_5(p^6)$ follows from computing $g_\alpha(p)$ for 
\[
\alpha \in
\left\{
(3,1,1,1), (2,2,1,1),(2,1,2,1),(2,1,1,2)
\right\}.
\]

There are some particular classes of compositions for which we can explicitly compute $g_\alpha(p)$.   We defer the proofs of the following two lemmas to Section \ref{sec:details}.
\begin{lemma}\label{Lem_beta}
Let $n \ge 2$ and $\alpha = (\beta,1,\ldots, 1)$ be a composition of length $n-1$.  
\begin{enumerate}
\item If $\beta = 2$, then $g_{\alpha}(p) = p^{n-2}$.
\item If $\beta \ge 3$, then $g_{\alpha}(p) = (n-1)p^{n-2}$.
\end{enumerate}
\end{lemma}
The first part of this lemma together with Lemma \ref{lem2} shows that 
\[
g_n(p^n) = g_{(2,1,\ldots,1)}(p) + g_{(1,2,1,\ldots,1)}(p) + \cdots + g_{(1,\ldots,1,2)}(p) =
p^{n-2} + p^{n-3} + \cdots + 1,
\] 
which proves the formula given in Proposition \ref{Propgnpn}.

\begin{lemma}\label{Lem_2beta}
Let $n \ge 3$ and $\alpha = (2,1,\ldots,1,\beta,1,\ldots, 1)$ be a composition of length $n-1$ where the $\beta$ is in the $k$\textsuperscript{th} position.  Set $r = n - 1-k$, i.e., the number of parts equal to $1$ following $\beta$.  Then
\begin{enumerate}
\item If $\beta = 2$, 
\[
g_{\alpha}(p) = p^{n-3+r} + (r+1) p^{n-3}(p-1).
\]
\item If $\beta \ge 3$, 
\[
g_{\alpha}(p) = (r+1)\left(p^{n-3+r} + p^{n-3}(p-1)\right).
\]
\end{enumerate}
\end{lemma}

Applying Lemma \ref{Lem_beta} and Lemma \ref{Lem_2beta} gives the following corollary.
\begin{cor}\label{cor_gnpn1}
For any $n \ge 1$ we have 
\begin{eqnarray*}
& & g_n(p^{n+1}) = \\
& &  \frac{2p^{2n-3} + (n^2-n)p^{n+1} - (n^2-n) p^n - (n^2-n+2)p^{n-1} + (n^2-n-2)p^{n-2}+2}{2(p-1)^2 (p+1)}.
\end{eqnarray*}
\end{cor}

\begin{proof}
By equation \eqref{gn_recursive} we have that 
\[
g_n(p^{n+1})   =  g_{n-1}(p^n) + g_{(3,1,\ldots,1)}(p) + g_{(2,2,1,\ldots,1)}(p) + g_{(2,1,2,\ldots,1)}(p) + g_{(2,1,\ldots,1,2)}(p).
\]
Lemma \ref{Lem_beta} implies that $g_{(3,1,\ldots,1)}(p) = (n-1)p^{n-2}$.  The first part of Lemma \ref{Lem_2beta} implies that
\begin{multline*}
g_{(2,2,1,\ldots,1)}(p) + g_{(2,1,2,\ldots,1)}(p) + \cdots + g_{(2,1,\ldots,1,2)}(p)  =    \\
p^{n-3} + p^{n-3+1} + \cdots + p^{n-3+n-3}  + p^{n-3}(p-1) \left(1 + 2 +\cdots +(n-2)\right) = \\
 p^{n-3}\left( \frac{p^{n-2}-1}{p-1} + (p-1) \binom{n-1}{2}\right).
\end{multline*}
We argue by induction, first noting that $g_1(p^2) = 0$ and 
\[
g_2(p^3) = \frac{ 2 p + 2 p^3 - 2 p^2 - 4 p + 2}{2(p-1)^2 (p+1)} = 1.
\]
The induction hypothesis along with some algebraic simplification shows that 
\[
g_n(p^{n+1})   - g_{n-1}(p^n) = (n-1)p^{n-2} + p^{n-3} \left( \frac{p^{n-2}-1}{p-1} + (p-1) \binom{n-1}{2}\right),
\]
completing the proof.

\end{proof}
\noindent Note that the formula of Corollary \ref{cor_gnpn1} matches the expressions for $g_5(p^6),\ g_6(p^7)$, and $g_7(p^8)$ given in Theorem \ref{gnpe_thm}.  The following result together with Lemmas \ref{lem2}, \ref{Lem_beta} and \ref{Lem_2beta} and the formula for $g_4(p^6)$ given in Table \ref{Tab1} implies the formula for $g_5(p^7)$ given in Theorem \ref{gnpe_thm}.

\begin{thm}\label{p7_comps}
We have that:
\begin{table}[h!]
\renewcommand{\arraystretch}{1.2}
\centering
\begin{tabular}{|c | c |}
\hline
$g_{(2,2,2,1)}(p)$ &  $3p^4+2p^3-4p^2$\\
$g_{(2,2,1,2)}(p)$ &  $p^4 +2p^3-2p^2$ \\
$g_{(2,1,2,2)}(p)$ &  $2p^3-p^2$ \\
$g_{(3,2,1,1)}(p)$ &  $7 p^4 - 6p^3 + 6p^2$ \\
$g_{(3,1,2,1)}(p)$ &  $p^4+5 p^3-2p^2$ \\
$g_{(3,1,1,2)}(p)$ &  $3p^3$\\
\hline
\end{tabular}.
\caption{{Values of $g_\alpha(p)$ for some $\alpha \in \mathcal{C}'_{5,7}$.}}
\end{table}
\end{thm}

\vspace{-.8cm}

\noindent We defer the proof to Section \ref{sec:details}.

In order to compute $g_5(p^8)$ and $g_6(p^8)$, we must work significantly harder.  The following result together with the results above imply the formulas for $g_5(p^8)$ and $g_6(p^8)$  given in Theorem \ref{gnpe_thm}.
\begin{thm}\label{p8_comps}
We have that:
\begin{table}[h!]
\renewcommand{\arraystretch}{1.2}
\centering
\begin{tabular}{|c | c |}
\hline
$g_{(2,1,2,3)}(p)$ &  $2p^3 - p^2$ \\
$g_{(2,1,3,2)}(p)$ &  $p^4$ \\
$g_{(2,2,1,3)}(p)$ &  $p^4 + 2p^3 - 2p^2$ \\
$g_{(2,2,2,2)}(p)$ &  $2p^4 - p^2$ \\
$g_{(2,2,3,1)}(p)$ &  $4p^4 + 2p^3 - 4p^2$ \\
$g_{(2,3,1,2)}(p)$ &  $3p^4 - p^2$\\
$g_{(2,3,2,1)}(p)$ &  $10p^4 - 11p^3 + 4p^2$\\
$g_{(3,1,1,3)}(p)$ &  $3p^3$\\
$g_{(3,1,2,2)}(p)$ &  $p^4 + 2p^3 - p^2$\\
$g_{(3,1,3,1)}(p)$ &  $2p^4 + 6p^3 - 2p^2$\\
$g_{(3,2,1,2)}(p)$ &  $5p^4 - 4p^3 + 2p^2$\\
$g_{(3,2,2,1)}(p)$ &  $p^5 + 8p^4 - 8p^3 + 4p^2$\\
$g_{(3,3,1,1)}(p)$ &  $15p^4 - 9p^3 + 3p^2$\\
$g_{(4,1,1,2)}(p)$ &  $p^4 + 2p^3$\\
$g_{(4,1,2,1)}(p)$ &  $5p^4 - p^3 - 2p^2$\\
$g_{(4,2,1,1)}(p)$ &  $16p^4 - 21p^3 + 6p^2$\\
\hline
\end{tabular},
\caption{{Values of $g_\alpha(p)$ for some $\alpha \in \mathcal{C}'_{5,8}$.}}
\end{table}\\
\noindent and 
\begin{table}[h!]
\renewcommand{\arraystretch}{1.2}
\centering
\begin{tabular}{|c | c |}
\hline
$g_{(2,1,1,2,2)}(p)$ &  $2p^4 - p^3$\\
$g_{(2,1,2,1,2)}(p)$ &  $p^5 + 2p^4 - 2p^3$ \\
$g_{(2,1,2,2,1)}(p)$ &  $3p^5 + 2p^4 - 4p^3$ \\
$g_{(2,2,1,1,2)}(p)$ &  $p^6 + 3p^4 - 3p^3$ \\
$g_{(2,2,1,2,1)}(p)$ &  $2p^6 + 2p^5 + 2p^4 - 7p^3 + 2p^2$ \\
$g_{(2,2,2,1,1)}(p)$ &  $4p^6 + 7p^5 - 16p^4 + 12p^3 - 6p^2$ \\
$g_{(3,1,1,1,2)}(p)$ &  $4p^4$\\
$g_{(3,1,1,2,1)}(p)$ &  $p^5 + 8p^4 - 4p^3$\\
$g_{(3,1,2,1,1)}(p)$ &  $8p^5 - 3p^4 + 3p^3$\\
$g_{(3,2,1,1,1)}(p)$ &  $5p^6 + 6p^5 - 14p^4 + 16p^3$\\
\hline
\end{tabular}.
\caption{{Values of $g_\alpha(p)$ for some $\alpha \in \mathcal{C}'_{6,8}$.}}
\end{table}
\end{thm}
\vspace{-1cm}

We do not give the details of the computations that go into the proof of Theorem \ref{p8_comps}.  They are similar to the computations for Theorem \ref{p7_comps}, but more extensive.

\subsection{$f_n(p^7)$ and $f_n(p^8)$: Details}\label{sec:details}  In this section we give the proofs of Lemma \ref{Lem_beta}, Lemma \ref{Lem_2beta}, and Theorem \ref{p7_comps}.

We first make some observations that will be useful throughout this section.  For the rest of this section, we write $v_1,\ldots, v_n$ for the columns of an $n \times n$ matrix $A$.  We write $\col(A)$ for the set of integer linear combinations of the columns of $A$.  If $A$ is a matrix in Hermite normal form with $\det(A) \neq 0$ then for any $i$ it is clear that $v_1 \circ v_i \in \col(A)$.  If $A$ is an irreducible subring matrix then its last column is $v_n = (1,1,\ldots,1)^T$ and for any $i$ it is clear that $v_i \circ v_n \in \col(A)$.

\begin{proof}[Proof of Lemma \ref{Lem_beta}]
An irreducible subring matrix $A$ that has diagonal $(p^\beta,p,\ldots, p,1)$ is of the form 
\[
\begin{pmatrix}
p^\beta & a_1 p   & a_2 p & \cdots & a_{n-2}p  &  1 \\
0   & p  & 0 & \cdots & 0 & 1 \\
0   & 0    & \ddots   & \cdots  & \vdots   & 1 \\
0   & 0    & 0   & \ddots  & \vdots   & 1 \\
0   & 0    & 0   & 0  & p   & 1 \\
0   & 0    & 0     & 0 & 0 & 1 \\
\end{pmatrix},
\]
where $0 \le a_i \le p^\beta -1$ for each $1\le i \le n-2$.  If $v_i \circ v_i \in \col(A)$, then $p^\beta \mid (a_i^2-a_i) p^2$.  If $v_i \circ v_j \in \col(A)$ for $2 \le i < j \le n-1$, then $p^{\beta} \mid a_i a_j p^2$.  Therefore, irreducible subring matrices with diagonal $(p^\beta,p,\ldots, p,1)$ are in bijection with solutions $(a_1,\ldots, a_{n-2})$ to 
\begin{eqnarray*}
a_i(a_i-1) & \equiv & 0 \pmod{p^{\beta-2}} \\
a_i a_j & \equiv & 0 \pmod{p^{\beta-2}}
\end{eqnarray*}
where we have one congruence of the first type for each $2\le i \le n-1$, and one congruence of the second type for each $2 \le i < j \le n-1$.

When $\beta =2$ any choice of $(a_1,\ldots, a_{n-2})$ gives an irreducible subring matrix.  There are $p$ choices for each $a_i$, completing the proof in this case.  

Suppose that $\beta \ge 3$.  Clearly $a_i(a_i-1)  \equiv  0 \pmod{p^{\beta-2}}$ implies that $a_i  \equiv 0 \pmod{p^{\beta-2}}$ or $a_i-1  \equiv 0 \pmod{p^{\beta-2}}$.  There are $p$ choices for $a_i$ with $a_i \equiv 0 \pmod{p^{\beta-2}}$ and $p$ choices for $a_i$ with $a_i -1 \equiv 0 \pmod{p^{\beta-2}}$.  If $i\neq j$ and $a_i-1 \equiv a_j-1 \equiv 0 \pmod{p^{\beta-2}}$, then $a_i a_j \not\equiv 0 \pmod{p^{\beta-2}}$, and $(a_1,\ldots, a_{n-2})$ does not give an irreducible subring matrix.

We see that $(a_1,\ldots, a_{n-2})$ gives an irreducible subring matrix if and only if either $a_i \equiv 0 \pmod{p^{\beta-2}}$ for each $i$, or there exists a unique $j$ with $a_j -1 \equiv 0 \pmod{p^{\beta-2}}$ and $a_i \equiv 0 \pmod{p^{\beta-2}}$ for every $i\neq j$.  This leads to $p^{n-2} + (n-2) p^{n-2} = (n-1) p^{n-2}$ total irreducible subring matrices.

\end{proof}

The next argument is similar, but more complicated.
\begin{proof}[Proof of Lemma \ref{Lem_2beta}]
An irreducible subring matrix $A$ that has diagonal $(p^2,p,\ldots, p,p^\beta,p,\ldots, p,1)$ where the $p^\beta$ is in the $k$\textsuperscript{th} column is of the form 
\[
\begin{pmatrix}
p^2 & a_1 p   & a_2 p & \cdots & a_{k-1} p & \cdots & \cdots &a_{n-2}p  &  1 \\
0   & p  & 0 & \cdots & 0 & \cdots & \cdots & 0 &  1 \\
0   & 0    & p   & \cdots  & \vdots   & \cdots & \cdots  & \vdots & 1 \\
0   & \cdots    & 0   & \ddots  & \vdots   &  \cdots & \cdots & \vdots & 1 \\
0   & \cdots    & \cdots   & 0  & p^\beta   & b_1 p & \cdots & b_r p &   1 \\
0   & \cdots    & \cdots     & \cdots & 0 & p & \cdots & \cdots  & 1 \\
0   & \cdots    & \cdots     & \cdots & \cdots & 0 & \ddots & \cdots  & 1 \\
0   & \cdots    & \cdots     & \cdots & \cdots & \cdots & 0 & p  & 1 \\
0   & \cdots    & \cdots     & \cdots & \cdots & \cdots &  \cdots & 0  & 1 
\end{pmatrix},
\]
where $0 \le a_i \le p -1$ for each $1\le i \le n-2$, and $0 \le b_j \le p^{\beta}-1$ for each $1 \le j \le r$.  It is easy to see that if $\min(i,j) \le k$, then $v_i \circ v_j \in \col(A)$.

For any $1\le m \le r,\ v_{k+m} \circ v_{k+m} \in \col(A)$ if and only if the following two congruences are satisfied:
\begin{eqnarray*}
b_m(b_{m}-1) & \equiv & 0 \pmod{p^{\beta-2}}, \\
a_{k-1} b_m(b_{m}-1) & \equiv & 0 \pmod{p^{\beta-1}}.
\end{eqnarray*}
For any $1\le i < j \le r$, we have $v_{k+i} \circ v_{k+j} \in \col(A)$ if and only if the following two congruences are satisfied:
\begin{eqnarray*}
b_i b_j & \equiv & 0 \pmod{p^{\beta-2}}, \\
a_{k-1} b_i b_j & \equiv & 0 \pmod{p^{\beta-1}}.
\end{eqnarray*}
We see that $a_1,\ldots, a_{k-2},a_{k},\ldots, a_{n-2}$ do not play a role in these congruences.  Therefore, the number of irreducible subring matrices with diagonal $(p^2,p,\ldots, p,p^\beta,p,\ldots, p,1)$ is equal to $p^{n-3}$ times the number of solutions $(a_{k-1},b_1,\ldots, b_r)$ to this collection of congruences.

We split the count into two pieces.  First suppose that $a_{k-1} \not\equiv 0 \pmod{p}$.  We count solutions to
\begin{eqnarray*}
b_m(b_{m}-1) & \equiv & 0 \pmod{p^{\beta-1}}, \\
b_i b_j & \equiv & 0 \pmod{p^{\beta-1}},
\end{eqnarray*}
where we have one inequality of the first type for each $1 \le m \le r$ and one inequality of the second type for each $1 \le i<j \le r$.  Since $0 \le b_m \le p^{\beta-1}-1$, there is a unique $b_m$ with $b_m \equiv 0 \pmod{p^{\beta-1}}$, and a unique $b_m$ with $b_m-1 \equiv 0 \pmod{p^{\beta-1}}$.  If $i\neq j$ and $b_i-1 \equiv b_j-1 \equiv 0 \pmod{p^{\beta-1}}$, then $b_i b_j \not\equiv 0 \pmod{p^{\beta-1}}$, and we do not get a solution to these congruences.  So, as in the proof of the previous lemma, we get $p^{n-3} (r+1) (p-1)$ irreducible subring matrices with $a_{k-1} \not\equiv 0 \pmod{p}$.

Now suppose that $a_{k-1} \equiv 0 \pmod{p}$.  We count solutions to
\begin{eqnarray*}
b_m(b_{m}-1) & \equiv & 0 \pmod{p^{\beta-2}}, \\
b_i b_j & \equiv & 0 \pmod{p^{\beta-2}},
\end{eqnarray*}
where we have one inequality of the first type for each $1 \le m \le r$ and one inequality of the second type for each $1 \le i<j \le r$. We are now in the exact same setting as in the proof of the previous lemma.  When $\beta = 2$ any of the $p^r$ choices of $(b_1,\ldots, b_r)$ gives a solution to these congruences.  This gives $p^{n-3+r}$ total irreducible subrings matrices.  When $\beta \ge 3$, there are $(r+1) p^r$ solutions $(b_1,\ldots, b_r)$ to these inequalities, which gives $(r+1) p^{n-3+r}$ total irreducible subrings with $a_{k-1} \equiv 0 \pmod{p}$.

\end{proof}

We next prove each of the formulas given in Theorem \ref{p7_comps}.  We give one case in detail and note that the remaining arguments are similar, but significantly easier.

\begin{lemma}\label{lemma3211}
We have $g_{(3, 2, 1, 1)}(p) =7p^4 - 6p^3 + 6p^2$. 
\end{lemma}

\begin{proof}
This is an easy computation when $p=2$, so for the rest of the proof suppose $p \ge 3$.

An irreducible subring matrix $A$ with diagonal $(p^3,p^2, p, p, 1)$ is of the form 
\[
\begin{pmatrix}
p^3 & cp    & x p & y p & 1 \\
0   & p^2  & u p & v p & 1 \\
0   & 0    & p   & 0   & 1 \\
0   & 0    & 0   & p   & 1 \\
0   & 0    & 0   & 0   & 1 \\
\end{pmatrix},
\]
where $0 \le c,x,y \le p^2-1,\ 0 \le u,v \le p-1$. 

If $v_2 \circ v_2 \in \col(A)$, then 
\[
\begin{pmatrix}
c^2 p^2 \\
p^4
\end{pmatrix}=
p^2\begin{pmatrix}
c p \\
p^2
\end{pmatrix}+
\lambda
\begin{pmatrix}
p^3 \\
0
\end{pmatrix}
\]
for some $\lambda\in\Z$.  This implies $p^3 \mid (c^2 p^2 - cp^3)$, so $p \mid c$. Define $c'$ by $c = p c'$ where $0 \le c' \le  p-1$.

Taking $v_3 \circ v_3$ or $v_4 \circ v_4$ and applying an argument like the one for $v_2\circ v_2$ gives 
\begin{eqnarray}
(x^2-x) - (u^2 - u) c' \equiv & 0 &  \pmod{p}\label{eq1} \\
(y^2-y) - (v^2 - v) c' \equiv & 0 & \pmod{p}\label{eq2}.
\end{eqnarray}

Taking $v_3 \circ v_4$ gives 
\begin{equation}\label{eq3}
xy -uv c' \equiv 0 \pmod{p}.
\end{equation}

These congruences depend only on $x$ and $y$ modulo $p$, rather than their particular values, so any solution $(x,y,u,v,c')$ to these three congruences gives $p^2$ irreducible subring matrices.  Therefore, we need only count solutions to equations \eqref{eq1}, \eqref{eq2}, and \eqref{eq3} for which $0 \le x, y \le p-1$.

If $c' = 0$, then equations \eqref{eq1} and \eqref{eq2} imply that $x,y \in \{0,1\}$.  By equation (\ref{eq3}) we cannot have $x=y=1$.  Any choices of $u$ and $v$ now satisfy these equations.  This gives $3p^2$ choices for $(x,y,u,v,c')$ and $3p^4$ irreducible subring matrices.

For the rest of the proof suppose $c' \neq 0$.  We consider cases based on $u$ and $x$.  Equation \eqref{eq1} implies that $u \in \{0,1\}$ if and only if $x \in \{0,1\}$.
\begin{claim}\label{ClaimA}
Suppose that $c' \neq 0$.  The following table gives the number of solutions to equations \eqref{eq1}, \eqref{eq2}, and \eqref{eq3} with specified values of $u$~and~$x$:
\[
\begin{tabular}{|c | c | c |}
\hline
$u$ & $x$ & Number of Solutions \\
\hline 
$0$ & $0$ & $p^2$ \\
\hline
$1$ & $0$ & $2(p-1)$ \\
\hline
$0$ & $1$ & $2(p-1)$ \\
\hline
$1$ & $1$ & $2(p-1)$ \\
\hline
 $\not\in\{0,1\}$ & $\not\in\{0,1\}$  & $3(p-2)^2$ \\
\hline
\end{tabular}\ .
\]
\end{claim}
\noindent We further divide up the last case of this claim.
\begin{claim}\label{ClaimB}
Suppose that $c' \neq 0$ and $u,x \not\in\{0,1\}$.  There are $(p-2)^2$ solutions with $v = 0$.  When $v \neq 0$ there are $(p-1)(p-2)$ solutions with $x =u$ and $(p-2)(p-3)$ solutions with $x \neq u$.
\end{claim}
Once these claims are established we count
\[
3p^4 + p^2 \left(p^2+6(p-1)+3(p-2)^2\right) = 7p^4-6p^3+6p^2
\]
total irreducible subrings, completing the proof.  

We now prove Claim \ref{ClaimA}.

\noindent \textbf{Case 1: $u=x=0$.}
  
We need only count solutions to equation \eqref{eq2}. If $v \in \{0, 1\}$, then for any of the $p-1$ choices for $c'$ there are $2$ solutions $y$, namely $y \in \{0, 1\}$, for a total of $4(p-1)$ solutions. Suppose $v\not\in\{0,1\}$.  For any $y \not\in \{0,1\}$ a unique value of $c'$ that gives a solution to this equation.  If $y \in \{0,1\}$ then we get no solutions.  Adding these cases together gives $4(p-1) + (p-2)^2 = p^2$ solutions.

\noindent \textbf{Case 2: $u=1,\ x=0$.}

Equation \eqref{eq3} implies $uvc' = 0$, and since $u$ and $c'$ are non-zero, we must have $v = 0$. Equation \eqref{eq2} implies $y \in \{0,1\}$, so accounting for the $p-1$ possible values of $c'$ gives $2(p-1)$ solutions.

\noindent \textbf{Case 3: $u=0,\ x=1$.}

Equation \eqref{eq3} implies $y = 0$. Equation (\ref{eq2}) gives $v \in \{0,1\}$, so accounting for the $p-1$ possible values of $c'$ gives $2(p-1)$ solutions.

\noindent \textbf{Case 4: $u=1,\ x=1$.}

Equation \eqref{eq3} gives $y \equiv v c' \pmod{p}$. Substituting this into equation \eqref{eq2} gives $(c'^2-c')v^2 \equiv 0 \pmod{p}$.  If $c' = 1$ we have $p$ choices for $v$.  If $c'\neq 1$ then $v = 0$.  This gives $p + p-2 = 2(p-1)$ solutions.

For the rest of the proof suppose that $c'\neq 0$ and $x,u \not\in \{0,1\}$.  We consider two further subcases.

\noindent \textbf{Case 5: $v=0$.}

Equation \eqref{eq3} implies $y=0$. Setting $c' = \frac{x^2-x}{u^2-u}$ for any choice of $x, u$ gives a valid solution.  This gives $(p-2)^2$ solutions. 

\noindent \textbf{Case 6: $v\neq0$.}

Equations \eqref{eq1} and \eqref{eq3} imply that
\[
c' \equiv \frac{x(x-1)}{u(u-1)} \equiv \frac{xy}{uv} \pmod{p}.
\]
This implies $v \equiv \frac{y(u-1)}{x-1} \pmod{p}$.  By assumption $v \neq 0$, so equation (\ref{eq3}) implies $y \neq 0$. Substituting this expression for $v$ into equation \eqref{eq2} and dividing by $y$ gives 
\begin{equation}\label{eq4}
y \left(1 - c'\frac{(u-1)^2}{(x-1)^2}\right) + c'\frac{u-1}{x-1}-1 \equiv 0 \pmod{p}.
\end{equation}
We need only count solutions to equation \eqref{eq4}.

Equation \eqref{eq4} is linear in $y$. The coefficient of $y$ is $0$ precisely when $c' \equiv \frac{(x-1)^2}{(u-1)^2} \pmod{p}$. Since $c' \equiv \frac{x(x-1)}{u(u-1)} \pmod{p}$ by equation (\ref{eq1}), this is equivalent to $\frac{x}{u} \equiv  \frac{x-1}{u-1} \pmod{p}$.  This implies $x \equiv u \pmod{p}$.

Suppose that $x = u$.  For any of the $p-2$ possible choices for $x$, any choice of $y$ gives a solution to this equation, except that $y = 0$ implies $v = 0$ by equation \eqref{eq3}, a case we have already considered.  Therefore, this case gives $(p-1)(p-2)$ solutions.

When $x\neq u$, for any of the $(p-2)(p-3)$ choices of $x$ and $u$ there are unique choices of $y$ and $c'$ such that equation \eqref{eq4} holds. This gives $(p-2)(p-3)$ solutions.

This completes the proofs of the two claims, which completes the proof of Lemma \ref{lemma3211}.
\end{proof}

\begin{proof}[Proof of Theorem \ref{p7_comps}]\ \\
\noindent\textbf{Case 1}: $g_{(2,2,2,1)}(p) = 3p^4+2p^3-4p^2$.  We count irreducible subring matrices $A$ of the form
\[
\begin{pmatrix}
p^2 & ap    & bp & cp & 1 \\
0   & p^2  & dp & ep & 1 \\
0   & 0    & p^2   & fp   & 1 \\
0   & 0    & 0   & p   & 1 \\
0   & 0    & 0   & 0   & 1 \\
\end{pmatrix},
\]
where $0\le a,b,c,d,e,f \le p-1$.  Taking pairwise products of columns shows that we get an irreducible subring matrix if and only if
\[
\begin{cases}
ad & \equiv 0 \pmod{p}, \\
d (f^2-f) & \equiv 0 \pmod{p}, \\
b (f^2-f) + a (e^2-e) & \equiv 0 \pmod{p}.
\end{cases}
\]
If $f \in \{0,1\}$, these congruences become $ad \equiv 0 \pmod{p}$ and $a (e^2-e) \equiv0 \pmod{p}$.  If $p\mid a$, then any values of $b,c,d,e$ give an irreducible subring. If $p \nmid a$ we must have $p \mid d$ and $e \in \{0,1\}$.  If $f \not\in \{0,1\}$ then we have $p \mid d$ and our final remaining condition is $b(f^2-f) + a (e^2-e) \equiv 0 \pmod{p}$.  For any choice of $a$ and $e$, there is a unique value of $b$ that satisfies this equation.  Therefore,
\[
g_{(2,2,2,1)}(p) = 2p^4+4(p-1)p^2 + (p-2)p^3 = 3p^4+2p^3-4p^2.
\]

\noindent\textbf{Case 2}: $g_{(2,2,1,2)}(p) = p^4+2p^3 - 2p^2$.  We count irreducible subring matrices $A$ of the form
\[
\begin{pmatrix}
p^2 & ap    & bp & cp & 1 \\
0   & p^2  & dp & ep & 1 \\
0   & 0    & p   & 0   & 1 \\
0   & 0    & 0   & p^2   & 1 \\
0   & 0    & 0   & 0   & 1 \\
\end{pmatrix},
\]
where $0\le a,b,c,d,e\le p-1$.  Taking pairwise products of columns shows that we get an irreducible subring matrix if and only if
\[
\begin{cases}
ae & \equiv 0 \pmod{p}, \\
a(d^2-d) & \equiv 0 \pmod{p}.
\end{cases}
\]
There are $p^4$ solutions with $p \mid a$, and $2p^2 (p-1)$ with $p \nmid a$.

\noindent\textbf{Case 3}: $g_{(2,1,2,2)}(p) = 2p^3-p^2$.  We count irreducible subring matrices $A$ of the form
\[
\begin{pmatrix}
p^2 & ap    & bp & cp & 1 \\
0   & p  & 0 & 0 & 1 \\
0   & 0    & p^2   & dp   & 1 \\
0   & 0    & 0   & p^2   & 1 \\
0   & 0    & 0   & 0   & 1 \\
\end{pmatrix},
\]
where $0\le a,b,c,d\le p-1$.  Taking pairwise products of columns shows that we get an irreducible subring matrix if and only if $bd \equiv 0\pmod{p}$.  There are $2p-1$ choices for the pair $(b,d)$ and any such pair gives $p^2$ irreducible subrings.

\noindent\textbf{Case 4}: $g_{(3,1,2,1)}(p) = p^4+5 p^3-2p^2$.  We count irreducible subring matrices $A$ of the form
\[
\begin{pmatrix}
p^3 & ap    & bp & cp & 1 \\
0   & p  & 0 & 0 & 1 \\
0   & 0    & p^2   & dp   & 1 \\
0   & 0    & 0   & p   & 1 \\
0   & 0    & 0   & 0   & 1 \\
\end{pmatrix},
\]
where $0\le a,b,c\le p^2-1$ and $0 \le d \le p-1$.   If $v_3 \circ v_3\in \col(A)$, then $p \mid b$.  Let $bp = b' p^2$ where $0 \le b' \le p-1$.  Taking pairwise products of columns shows that we get an irreducible subring matrix if and only if 
\[
\begin{cases}
a^2-a &\equiv 0\pmod{p}, \\
ac &\equiv 0\pmod{p},\\
(c^2-c) - b'(d^2-d)&\equiv 0\pmod{p}.
\end{cases}
\]
These congruences only depend on $a,c$ modulo $p$, so we count solutions modulo $p$ and then multiply by $p^2$.  First suppose that $p \mid a$.  It is easy to see that there are $2(3p-2) + (p-2)^2$ solutions to $(c^2-c) - b'(d^2-d) \equiv 0 \pmod{p}$.  If $p \nmid a$ then $p \mid c$ and we need only note that there are $3p-2$ solutions to $b'(d^2-d) \equiv 0 \pmod{p}$.  Combining these observations completes the proof.

\noindent\textbf{Case 5}: $g_{(3,1,1,2)}(p) = 3p^3$.  We count irreducible subring matrices $A$ of the form
\[
\begin{pmatrix}
p^3 & ap    & bp & cp & 1 \\
0   & p  & 0 & 0 & 1 \\
0   & 0    & p   & 0   & 1 \\
0   & 0    & 0   & p^2   & 1 \\
0   & 0    & 0   & 0   & 1 \\
\end{pmatrix},
\]
where $0\le a,b,c\le p^2-1$.  If $v_4 \circ v_4 \in \col(A)$, then $p \mid c$.  Let $cp = c' p^2$ where $0 \le c' \le p-1$.  Taking pairwise products of columns shows that we get an irreducible subring matrix if and only if
\[
\begin{cases}
a^2-a &\equiv 0\pmod{p}, \\
b^2-b &\equiv 0\pmod{p}, \\
ab &\equiv 0\pmod{p}.
\end{cases}
\]
These congruences are satisfied if and only if $(a,b) \pmod{p} \in \{(0,0),(0,1),(1,0)\}$.  This gives $3p^3$ total irreducible subrings.

\end{proof}

We do not give details for the $g_\alpha(p)$ computations of Theorem \ref{p8_comps}.  They are similar in spirit to the computations of this section but the details are significantly more extensive.

\section{Lower Bounds on $g_n(p^e)$}\label{sec_lower_bounds}

We now give a lower bound on $g_n(p^e)$ when $n-1 \le e \le 2(n-1)$.  We do this by giving a lower bound on $g_{\alpha}(p^e)$ for a particular composition $\alpha$ of $e$ of length $n-1$.  These lower bounds on $g_n(p^e)$ together with Proposition \ref{fnrecurrence} give lower bounds on $f_n(p^e)$.

\begin{prop}\label{lowerbound}
Let $\alpha = (2,\ldots, 2,1,\ldots, 1)$ be a composition of length $n-1$ with $r$ entries equal to $2$ and $s$ entries equal to $1$.  Then $g_{\alpha}(p) \ge p^{rs}$.
\end{prop}
\noindent Note that $r+s = n-1$ and $2r + s = e$.  Solving for $r$ and $s$ in terms of $n$ and $e$ gives $(r,s) = \left(e-(n-1), 2(n-1)-e\right)$.

\begin{proof}
Let $A$ be an upper triangular matrix with columns $v_1,\ldots, v_n$ where the diagonal entry of columns $v_1,\ldots, v_r$ is $p^2$, the diagonal entry of columns $v_{r+1},\ldots, v_{r+s}$ is $p$, and the final column is $(1,\ldots, 1)^T$.  Suppose that every non-diagonal entry in the first $n-1$ columns of this matrix is zero except possibly in the first $r$ rows of columns $v_{r+1},\ldots, v_{r+s}$.  In each of these $rs$ entries there are $p$ integers $a_{i,j}$ satisfying $0 \le a_{i,j} \le p^2-1$ and $a_{i,j} \equiv 0 \pmod{p}$.  This gives $p^{rs}$ total matrices. It is easy to check that each one is an irreducible subring matrix.
\end{proof}
The proof of this proposition comes from producing a set of irreducible subring matrices with diagonal $(p^2,\ldots, p^2, p,\ldots, p, 1)$ that have many entries equal to $0$.  Not all irreducible subring matrices with this diagonal have all of these entries equal to $0$, so the lower bound of Proposition \ref{lowerbound} is not actually equal to $g_{\alpha}(p)$. For example, Lemma \ref{Lem_2beta} implies that $g_{(2, 2, 1, 1)}(p) = p^4 + 3p^2(p-1)$, larger than the lower bound of $p^4$ from Proposition \ref{lowerbound}.

Proposition \ref{lowerbound} gives a lower bound on $g_n(p^e)$ for every pair $(n,e)$ with $n-1 \le e \le 2(n-1)$.  We now determine for a fixed value of $e$, which $n$ gives the largest lower bound.  That is, for fixed $e$ we want the non-negative integer $n$ maximizing the function $h(e,n) = (e-(n-1))(2(n-1)-e)$.  It is easy to check that $h(e,n+1) \ge h(e,n)$ if and only if $n \le \frac{3e+2}{4}$.  This gives the following lower bound.

\begin{cor}\label{Cor1}
Let $e$ be a positive integer.  We have
\[
\max_n g_n(p^e) \ge
\begin{cases} 
p^{\frac{e^2}{8}} & \text{if } e \equiv 0 \pmod{4} \\
p^{\frac{1}{8}(e^2-1)} & \text{if } e \equiv 1 \pmod{4} \\
p^{\frac{1}{8}(e^2-4)} & \text{if } e \equiv 2 \pmod{4} \\
p^{\frac{1}{8}(e^2-1)} & \text{if } e \equiv 3 \pmod{4}
\end{cases}.
\]
\end{cor}

In Section \ref{sec3} we saw that for each $e \le 8$ and each $n,\ g_n(p^e)$ was given by a polynomial in $p$.  For each $e$, the maximum over all $n$ of the degree of this polynomial is equal to the degree of the monomial on the right hand side of the expression given in Corollary \ref{Cor1}.  For example, 
\[
g_7(p^8) = p^8 + p^7 + 2p^6 + 23p^5 + 3p^4 + 2p^3 + 2p^2 + p + 1,
\]
a polynomial of degree $\frac{8^2}{8} =8$.  It is unclear whether for larger values of $e$ these lower bounds will continue to grow at a rate similar to the growth of $\max_n g_n(p^e)$. 

The lower bounds of this section are related to Brakenhoff's lower bounds for orders of bounded index in the ring of integers of a number field.
\begin{prop}\cite[Lemma 5.10]{Brakenhoff}\label{BrakenhoffLower}
Let $\mathcal{O}_K$ be the ring of integers of a number field $K$.  Every additive subgroup $G$ of $\mathcal{O}_K$ that satisfies $\Z + m^2 \mathcal{O}_K \subset G \subset \Z + m \mathcal{O}_K$ for some integer $m$ is a subring.
\end{prop}  
The subrings $R$ described in the proof of Proposition \ref{lowerbound} do satisfy $\Z + p^2 \Z^n \subset R \subset \Z+ p \Z^n$, where the first $\Z$ is interpreted as integer multiples of the multiplicative identity $(1,\ldots, 1)$.  Brakenhoff gives a lower bound for the number of additive subgroups satisfying the hypothesis of Proposition \ref{BrakenhoffLower} and derives a lower bound for the number of orders of index at most $X$ in the ring of integers of a degree $n$ number field $K$.  This requires an easy optimization along the lines of Corollary \ref{Cor1}.  In this way, our lower bounds for $g_n(p^e)$ are analogous to the lower bounds from \cite[Theorem 5.1]{Brakenhoff}. 

\section{Further Questions}\label{conjectures}

\subsection{Uniformity of $\zeta_{\Z^n}^R(s)$ and Varieties over Finite Fields}
Questions \ref{zetaznuniform} and \ref{fnpolyq} are about how counting functions vary with $p$.  Theorem \ref{GSSthm} gives  information on how $f_n(p^e)$ behaves for fixed $n$ and $p$.

We recall a theorem of du Sautoy and Grunewald \cite{duSautoyGrunewald}, about the behavior of local factors of zeta functions of rings as we vary the prime.  This result follows from the machinery of \cite{duSautoyGrunewald}, suitably modified for the requirement that subrings must contain the multiplicative identity.  We use the notation of \cite[Theorem A]{Voll2}.

\begin{thm}\label{dSGthm}
Let $L$ be a ring of additive rank $n$ containing a multiplicative identity.  Then there are smooth projective varieties $V_t,\ t\in \{1,\ldots, m\}$, defined over $\Q$, and rational functions $W_t(X,Y) \in \Q(X,Y)$ such that for almost all primes $p$ the following holds:\\
Denoting by $b_t(p)$ the number of $\F_p$-rational points of $\overline{V_t}$, the reduction $\mod p$ of $V_t$, we have
\[
\zeta^R_{L,p}(s) = \sum_{t=1}^m b_t(p) W_t(p,p^{-s}).
\]
\end{thm}
Not much is known about the types of varieties that can appear in these zeta functions as we vary over different rings.  See the paper of du Sautoy \cite{duSautoyDenom} and Voll's survey \cite[Section 2.1]{Voll} for more information.

In case $\zeta_{\Z^n}^R(s)$ is not uniform it would be interesting to see what kinds of varieties arise in the formulas of Theorem \ref{dSGthm}.  The conditions for the columns of an $n\times n$ matrix to generate a multiplicatively closed sublattice of $\Z^n$ define many equations in the matrix entries.  For examples for $n= 4$ and $5$, see \cite[Lemmas 12 and 13]{KMTB}.  It is possible that once $n$ and $e$ are large enough, varieties $V_t$ occur for which the functions $b_t(p)$ in Theorem \ref{dSGthm} are not polynomials in $p$, and that these functions occur in formulas for $f_n(p^e)$.

\subsection{Coefficients of $f_n(p^e)$ and $g_n(p^e)$}

For small fixed values of $n$ and $e$, the function $g_n(p^e)$ is a polynomial in $p$ with non-negative coefficients. However, this is not true for $g_5(p^8) = p^5 + 77p^4 - 13p^3 + 12p^2 + p + 1$.  As far as we know, there has been no previous study of the positivity of coefficients of $g_n(p^e)$ or $f_n(p^e)$.  These questions are motivated by analogous work related to Hall polynomials.  

\begin{definition}
\begin{enumerate}[wide, labelwidth=!, labelindent=0pt]  
\item Let $\lambda = (\lambda_1, \ldots, \lambda_k)$, where $\lambda_1 \ge \cdots \ge \lambda_k > 0$.  A finite abelian $p$-group $G$ is of \emph{type $\lambda$} if 
\[
G \cong \Z/p^{\lambda_1} \Z \times \cdots \times \Z/p^{\lambda_k} \Z.
\]

\item A subgroup of $H$ of $G$ is of cotype $\nu$ if $G/H$ is of type $\nu$. 

\item Let $g_{\mu\nu}^{\lambda}(p)$ be the number of subgroups $H$ of a finite abelian $p$-group $G$ of type $\lambda$ such that $H$ has type $\mu$ and cotype $\nu$.  
\end{enumerate}
\end{definition}
Hall proved that $g_{\mu\nu}^{\lambda}(p)$ is a polynomial in $p$ with integer coefficients.  Several other authors have studied these coefficients.  For example, Butler and Hales give a characterization of types~$\lambda$ for which all of the associated Hall polynomials have non-negative coefficients \cite{ButlerHales}.  

Maley shows that the expansion of any $g_{\mu\nu}^{\lambda}(p)$ in terms of powers of $p-1$ has non-negative coefficients \cite{Maley}.  In all cases we have computed, the same property holds for $g_n(p^e)$.  This is stronger than the observation that $g_n(1)$ is a non-negative integer.  

\begin{question}
When $g_n(p^e)$ is expanded in terms of powers of $p-1$, are the coefficients positive?
\end{question}
For an example of non-negativity questions like this for zeta functions associated to graphs, see the recent work of Rossmann and Voll \cite[Section 1.9]{RossmannVoll}.  Evseev has studied the substitution $p = 1$ in the form of the \emph{reduced zeta function} \cite{Evseev}.  The $p \to 1$ behavior of local factors of zeta functions is related to the corresponding topological zeta function \cite{Rossmann2}.  It would be interesting to undertake a more detailed study of the coefficients of $f_n(p^e)$ and $g_n(p^e)$.  For more background on Hall polynomials and connections to counting subgroups of finite abelian groups, see the books of Macdonald \cite{Macdonald} and Butler \cite{ButlerBook}.

\section*{Acknowledgements}

We thank the mathematics department at Yale University and the Summer Undergraduate Research at Yale (SUMRY) program for providing the opportunity to conduct this research. SUMRY is supported in part by NSF grant CAREER DMS-1149054. The second author was supported by NSF Grant DMS 1802281, NSA Young Investigator Grant H98230-16-10305 and an AMS-Simons Travel Grant.

We thank Franco Williams, for his active involvement throughout this project and for many helpful conversations. We thank Christopher Voll and Tobias Rossmann for many extremely helpful comments.  We would also like to extend our gratitude to Sam Payne, Sam Kimport, and Jos\'e Gonz\'alez for helping to organize SUMRY. The second author thanks Kelly Isham and Robert Lemke Oliver for helpful conversations and computational assistance.  We thank the referee for many helpful comments that greatly improved the paper.  Lastly, we thank the Yale Center for Research Computing for High Performance Computing resources.

\end{document}